\documentclass[12pt,bezier]{article}
\usepackage{times,tikz}
\usepackage{booktabs}
\usepackage{pifont}
\usepackage{floatrow}
\usepackage{caption}
\usepackage{mathrsfs}
\usepackage[fleqn]{amsmath}
\usepackage{amsfonts,amsthm,amssymb,mathrsfs,bbding}
\usepackage{txfonts}
\usepackage{graphics,multicol}
\usepackage{graphicx}
\usepackage{color}
\usepackage{enumerate}
\usepackage{caption}
\usepackage{cite}
\usepackage{latexsym,bm}
\usepackage{indentfirst}
\usepackage{mathtools}
\evensidemargin=\oddsidemargin\topmargin=-1.5cm

\newtheorem{thm}{Theorem}[section]

\newtheorem{lem}{Lemma}[section]
\newtheorem{cor}{Corollary}[section]

\theoremstyle{definition}

\addtocounter{section}{0}
\begin{document}
\title{A relation between the multiplicity of nonzero eigenvalues and the induced matching number of graph \footnote{This work is supported
by National Natural Science Foundation of China (No. 11371372).}}
\author{{Qian-Qian Chen, \,\, Ji-Ming Guo}\footnote{Corresponding author.}\setcounter{footnote}{-1}\footnote{\emph{Email address:}qianqian\_chen@yeah.net(Q,-Q. Chen), jimingguo@hotmail.com(J.-M. Guo).}\\[2mm]
\small School of Mathematics, \\
\small  East China University of Science and Technology, Shanghai, P. R. China.}

\date{}
\maketitle

{\flushleft\large\bf Abstract}
Let $G$ be a graph with an adjacency matrix $A(G)$. The multiplicity of the eigenvalue $\lambda$ of $A(G)$ is denoted by $m_\lambda(G)$. In \cite{Wong}[D. Wong, J. Wang, J. H. Yin, A relation between multiplicity of nonzero eigenvalues of trees and their matching numbers, Linear Algebra Appl. 660 (2023) 80--88.], the author apply the Pater-Wiener Theorem to prove that  $m_\lambda(T)\leq \beta'(T)$ for any $\lambda\neq0$. Moreover, they characterized all trees with $m_\lambda(T)=\beta'(T)$ and $m_\lambda(T)=\beta'(T)-1~(\beta'(T)\geq2)$, where $\beta'(T)$ denotes the induced matching number of $T$.

In this paper, we extend this result from trees to any connected graph. For any non-zero eigenvalue $\lambda$ of the connected graph $G$, we prove that $m_\lambda(G)\leq \beta'(G)+c(G)$, where $c(G)$ represents the cyclomatic number of $G$. The equality holds if and only if $G\cong C_3(a,a,a)$ or $G\cong C_5$, or a tree with the diameter is at most $3$. Furthermore, if $\beta'(G)\geq3$, we characterize all connected graphs with $m_\lambda(G)=\beta'(G)+c(G)-1$.
\vspace{0.1cm}
\begin{flushleft}

\textbf{Keywords:} Graph; eigenvalue multiplicity; induced matching number;matching number.
\end{flushleft}
\textbf{AMS Classification:} 05C50

\section{Introduction}\label{s-1}
All graphs presented in this paper are simple undirected graphs. Let $G=(V(G), E(G))$ be a graph with vertex set $V(G)$ and an edge set $E(G)$. The \emph{adjacency matrix} of $G$ is defined as the matrix $A(G)=(a_{uv})$, where $a_{uv}=1$ if $u,~v$ are adjacent in $G$, and $0$ otherwise. The eigenvalues of $A(G)$ are referred to as the eigenvalues of $G$. We denote the multiplicity of a real number $\lambda$ as an eigenvalue of $G$ by $m_\lambda(G)$, and the set of all eigenvalues in $A(G)$ by $\sigma(G)$.

A graph $H$ is a \emph{ subgraph } of $G$ if $V(H)\subseteq V(G)$ and $E(H)\subseteq E(G)$. Moreover, $H$ is called an \emph{induced subgraph} of $G$ if two vertices of $V(H)$ are adjacent in $H$ if and only if they are adjacent in $G$. For a subset $U$ of $V(G)$,  the graph obtained by removing the vertices in $U$ along with all the edges incident to them from $G$  is denoted as $G-U$. Similarly, if $H$ is a subgraph of $G$ and $U\in V(G)\setminus V(H)$, then the subgraph induced by the vertex set $V(H)\cup U$ in $G$ denoted as $G+U$. We use $C_n(\text{respectively}, K_{1, n-1})$ to represent a cycle (\text{respectively}, a star) with $n$ vertices. The \emph{cyclomatic number} of a graph $G$, is defined as $c(G)=|E(G)|-|V(G)|+\omega(G)$, where $\omega(G)$ indicates the number of connected components in $G$. If $G$ is connected, then $G$ is a tree if $c(G)=0$, or is a unciyclic graph if $c(G)=1$.

Given a graph $G$, a \emph{matching} $M$ in $G$ is a set of pairwise non-adjacent edges, i.e., no two edges have a common vertex. A \emph{maximum matching} is a matching  that contains the largest number of edges possible. The size of a maximum matching for a graph $G$ is known as the \emph{matching number} of $G$ and is represented by $\beta(G)$. If no edge in $G$ connects endpoints of distinct edges in $M$, we say that $M$ is an \emph{induced matching} of $G$, denoted by $M^*$.  If a vertex $v$  in $G$ is matched by an induced matching $M^*$, we say that $v$ is \emph{$M^*$-saturated}, otherwise, $v$ is \emph{$M^*$-unsaturated}. We denote the number of edges in a maximum induced matching of $G$ by $\beta'(G)$. It is clear that for any graph $G$, $\beta'(G)\leq\beta(G)$. For example, $\beta(P_n)=\lfloor\frac{n}{2}\rfloor$ and $\beta'(P_n)=\lfloor\frac{n+1}{3}\rfloor$ for $n\geq2$.

The matching number is a significant structural parameter in graph theory, and it is closely linked to the eigenvalue multiplicities of a graph.  If $\lambda=0$, Cvetkovi\'{c} et al. \cite{CV} showed that $m_0(T)=|V(T)|-2\beta(T)$ holds for any tree $T$. In 2014, Wang et al. \cite{Wang2} provided both an upper and a lower bound for $\lambda=0$ in any graph $G$ based on the matching number and cyclomatic number, which are $n(G)-2 m(G)-c(G)\leq m_0(G)\leq n(G)-2 m(G)+2 c(G)$.
Furthermore, the characterization of graphs  with $m_0(G)=n(G)-2 m(G)-c(G)$ was presented in \cite{Wang1}, while the characterization of graphs with $m_0(G)=n(G)-2 m(G)+2 c(G)$ was given in \cite{Song}. Recently, Ma et al. \cite{Ma} improved the lower bound  and established the inequality $m_0(G)\geq n(G)-2 m(G)-\theta(G)$, where $\theta(G)$ is the least number of edges deleted from $G$ to make $G$ to be a bipartite connected graph.  Building upon these findings, Zhou et al.\cite{Zhou} further refined the bounds by considering the matching number, the maximum number of disjoint odd cycles in $G$ , as well as the number of even cycles present in $G$.  In \cite{Row}, Rowlinson showed that, if a tree $T$ is not $P_2$ or $Y_6$ and it has $1$ as an eigenvalue of multiplicity $k$, then it has $k+1$ pendant edges that form an induced matching, here $Y_6$ is the unique tree of order $6$ with two adjacent vertices of degree $3$. Recently, Wong et al. \cite{Wong} extended Rowlinson's result by replacing eigenvalue $1$ with an arbitrary nonzero eigenvalue. Additionally, they also describe all the trees attaining the upper bound, as follows:
\begin{thm}\label{upper-bound}\cite{Wong}
Let $T$ be a tree with $\lambda \neq 0$ as an eigenvalue of multiplicity $k \geq 1$. If the diameter of $T$ is at least $4$, then $T$ has $k+1$ pendant edges forming an induced matching of $T$. Particularly, $m_\lambda(T) \leq \beta^{\prime}(T)$, and $m_\lambda(T) \leq \beta^{\prime}(T)-1$ if the diameter of $T$ is at least $4$.
\end{thm}
Let $diam(G)$ denote the diameter of $G$. A \emph{caterpillar graph} is a tree which on removal of all its pendant vertices leaves a path. In such a graph, the path is called the \emph{backbone}, and the pendant edges are called \emph{hairs}. In \cite{Wong}, Wong et al. also provide a characterize of all trees with a nonzero eigenvalue $\lambda$ having a multiplicity of $\beta^{\prime}(T)-1$, which is stated below:
\begin{thm}\label{extremal-tree}\cite{Wong}
Let $T$ be a tree with $\lambda$ as a nonzero eigenvalue. Then $m_\lambda(T)=\beta^{\prime}(T)-1$ if and only if $T$ is one of the following graphs:
\begin{enumerate}[(a)]
\vspace{-0.2cm}
\item $T$ is a caterpillar graph with diameter $4$, $5$ or $6$ , and the center vertex of the backbone of $T$ has degree $2$ if the diameter is $6$.
\vspace{-0.2cm}
\item There is a vertex $w$ such that $T-w=H_1 \cup \ldots \cup H_{s+1} \cup I$, where $s+1 \geq 3$, $I$ induces a null graph (without edges) and each component $H_i$ of $T-w$ has $\lambda$ as an eigenvalue and $\operatorname{diam}\left(H_i\right) \leq 3$. Moreover, either $\operatorname{diam}\left(H_i+w\right)=2$ for some $i$, or $I=\emptyset$ and $\operatorname{diam}\left(H_i+w\right)=3$ for all $i$.
\vspace{-0.2cm}
\end{enumerate}
\end{thm}
In this paper, we intend to extend this result from trees to any connected graph. we prove that for any non-zero eigenvalue $\lambda$ of the connected graph $G$, $m_\lambda(G)\leq \beta'(G)+c(G)$, where $c(G)$ is the cyclomatic number of $G$. Equality holds if and only if $G\cong C_3(a,a,a)$ or $G\cong C_5$. Furthermore, if $\beta'(G)\geq3$, we characterize all connected graphs with $m_\lambda(G)=\beta'(G)+c(G)-1$.
\section{Preliminaries}\label{s-2}
Let $A=\left(a_{i j}\right)_{n \times n}$ be the adjacency matrix of $G$. For any eigenvalue $\lambda$ of $A$, the eigenspace of $\lambda$ is defined as $\mathcal{E}(\lambda)=\left\{x \in \mathbb{R}^n: A x=\lambda x\right\}$. Let $\left\{\mathbf{e}_1, \mathbf{e}_2, \ldots, \mathbf{e}_n\right\}$ denote the standard orthonormal basis, and let $E_\lambda$ denote the matrix which represents the orthogonal projection of $\mathbb{R}^n$ onto the eigenspace $\mathcal{E}(\lambda)$ of $A$ with respect to $\left\{\mathbf{e}_1, \mathbf{e}_2, \ldots, \mathbf{e}_n\right\}$. There always exists $X \subseteq\{1, \ldots, n\}$ such that vectors $E_\lambda \mathbf{e}_i(i \in X)$ form a basis for $\mathcal{E}(\lambda)$, such a set $X$ is called a star set for eigenvalue $\lambda$ of $A$ (see \cite{Cv1}). Clearly $|X|=\operatorname{dim} \mathcal{E}(\lambda)$ is the multiplicity of $\lambda$.  Since $E_\lambda$ is a polynomial function of $A$, we have
\begin{equation}\label{eigenvalue-equ}
\lambda E_\lambda \mathbf{e}_i=\sum_{j \sim i}  E_\lambda \mathbf{e}_j,
\end{equation}
where $j \sim i$ means that $i$ and $j$ are adjacent in graph $G$.
\begin{lem}\label{interlacing}\cite{Cv2}
Let $G$ be a graph with a vertex $v$. Then $m_\lambda(G-v)-1\leq m_\lambda(G)\leq m_\lambda(G-v)+1$.
\end{lem}
We will use the following lemma in our proof of the main results. Its proof can be seen in \cite{CGW}, but we give the proof for completeness.
\begin{lem}\label{GuvH}
Let $GuvH$ be the graph obtained from $G \cup H$ by adding an edge joining the vertex $u$ of $G$ to the vertex $v$ of $H$. For any eigenvalue $\lambda$ of $A(GuvH)$, if  $m_\lambda (G)=m_\lambda (G-u)+1$, then $m_\lambda(GuvH)=m_\lambda(H-v)+m_\lambda (G)-1$.
\end{lem}
\begin{proof}
Let $A(GuvH)$ defined as follows:
\begin{equation*}
 A(GuvH)=\left(\begin{array}{cccc}
 A(G-u) & \alpha & O & O\\
\alpha^T &0 & 1&O \\
  O & 1 &0&\beta^T\\
  O&O&\beta&A(H-v)\\
 \end{array}\right).
 \end{equation*}
Then
 \begin{equation*}
 A(GuvH)-\lambda I=\left(\begin{array}{cccc}
 A_1 & \alpha & O & O\\
\alpha^T &-\lambda & 1&O \\
  O & 1 &-\lambda&\beta^T\\
  O&O&\beta&A_2\\
 \end{array}\right),
 \end{equation*}
where $A_1=A(G-u)-\lambda I$ and $A_2=A(H-v)-\lambda I$. Since $m_\lambda (G-u)=m_\lambda (G)-1$, it is easy to see $(\alpha^T, -\lambda)$ can be represented linearly by the row vectors of $[A_1 , \alpha]$. Therefore
\begin{equation*}
 A(GuvH)-\lambda I=\left(\begin{array}{cccc}
 A_1 & \alpha & O & O\\
\alpha^T &-\lambda & 1&O \\
  O & 1 &-\lambda&\beta^T\\
  O&O&\beta&A_2\\
 \end{array}\right)\rightarrow
 \left(\begin{array}{cccc}
 A_1 & O & O & O\\
O&0 & 1&O \\
  O & 1 &-\lambda&\beta^T\\
  O&O&\beta&A_2\\
 \end{array}\right)=A'.
 \end{equation*}
 Let
 \begin{equation*}
Q=\left(\begin{array}{cccc}
I & O & O&O\\
O &1& 1&-\beta^T\\
  O & 0 & 1&0\\
  O&O&O&I\\
 \end{array}\right),
 \end{equation*}
 we have
 \begin{equation*}
 Q^TA'Q=\left(\begin{array}{cccc}
A_1 & O & O & O\\
O &0& 1&O \\
  O & 1 & 2-\lambda&O\\
  O&O&O&A_2\\
 \end{array}\right),
 \end{equation*}
 which implies that $rank(A')=rank(A_1)+2+rank(A_2)$. Note that $rank(A(GuvH)-\lambda I)=rank(A')$,  we have
 \begin{equation*}
 \begin{array}{lll}
 m_\lambda(GuvH)&=&m_\lambda(A_1)+m_\lambda(A_2)\\
 &=&m_\lambda(G-u)+m_\lambda(H-v)\\
 &=&m_\lambda (G)-1+m_\lambda(H-v),
 \end{array}
 \end{equation*}
as required.
\end{proof}

%The double star graph, which is denoted as $S_{a,b}$, can be obtained by connecting the center vertices of two star graphs. In addition, the graph $C_3(a,a,a)$ can be obtained by adding $a$' pendant vertices to each vertex of $C_3$, as shown in Figure \ref{fig-star}.
\begin{lem}\cite{Wong}\label{star}
Let $T$ be a tree with diameter $d\leq3$. If $\lambda\neq0$ is an eigenvalue of $T$, then $\lambda$ is not an eigenvalue of $T-v$ for any vertex $v$ of $T$.
\end{lem}
\begin{lem}\cite{Wong}\label{m(T)=b(T)}
Let $T$ be a tree with $\lambda$ as a nonzero eigenvalue. Then $m_\lambda(T)=\beta'(T)$ if and only if the diameter of $T$ is $1$, $2$ or $3$.
\end{lem}
\begin{lem}\cite{Wong}\label{caterpillar}
Each nonzero eigenvalue of caterpillar graph is simple.
\end{lem}

\begin{lem}\label{C3}
Consider a graph $G$ that is  isomorphic to either $C_3(a, a, a)$  or $C_5$, where $C_3(a,a,a)$  is a graph  obtained by adding $a$' pendant vertices to each vertex of $C_3$. Let $\lambda$ be a non-zero eigenvalue of $A(G)$. If $m_\lambda(G)=2$, then $m_\lambda(G-x)=m_\lambda(G)-1=1$ for any $x\in V(G)$.
\end{lem}
\begin{proof}
We can easily  verify that $C_5$ holds. Next, we consider the case where $G$ is isomorphic to $C_3(a, a, a)$.  We will prove this by contradiction. If $x$ is a vertex on $C_3$ and $m_\lambda(G)\leq m_\lambda(G-x)$, then by Lemma \ref{caterpillar}, we have
$$m_\lambda(G)\leq m_\lambda(G-x)=1,$$
which leads to a contradiction. Therefore, according to Lemma \ref{interlacing}, $m_\lambda(G-x)=m_\lambda(G)-1=1$. If $x$ is a pedant vertex of $C_3(a,a,a)$ and is adjacent to $y$, then based on Equ.$(\ref{eigenvalue-equ})$, we have $\lambda E_\lambda e_x=E_\lambda e_y$. As $\lambda\neq0$ and $y$ belongs to a star set of $G$, it follows that $x$ also belongs to a star set of $G$, as required.
\end{proof}

\begin{lem}\label{C5}
Let $G$ be the graph obtained by connecting any two vertices in $C_3(a,a,a)~(\text{or}~C_5)$ with the central vertex of $K_{1,s}$. Then, for any non-zero eigenvalue $\lambda$ of $A(G)$, $m_\lambda(G)\leq2$.
\end{lem}
\begin{proof}
We will prove by contradiction. If $m_\lambda(G)\geq3$, then we can construct a non-zero eigenvector $Y$ of $A(G)$ corresponding to $\lambda$ such that there are suitable two adjacent vertices , say $u$ and $v$, in the cycle of $C_3(a,a,a)~(\text{or}~C_5)$ such that $Y(u)=Y(v)=0$. However, according to the eigen-equation $A(G)Y=\lambda Y$, we have $Y=\textbf{0}$, a contradiction.
\end{proof}
\section{Main Result}\label{s-4}

\begin{thm}\label{thm-upper-bound}
Let $G$ be a connected graph with $\lambda \neq 0$ as an eigenvalue. Then $m_\lambda(G) \leq \beta^{\prime}(G)+c(G)$, and the equality hold if and only if $G\cong C_3(a, a, a)~(a\geq0)$ or $G\cong C_5$, or a tree with the diameter $1$, $2$ or $3$.
\end{thm}
\begin{proof}
Note that for any cycle of $G$, there exists at least one vertex that is not yet covered by the induced matching $M^*$. Therefore, we can choose any one $M^*$-unsaturated vertex, denoted by $x_i$~($1\leq i\leq c(G)$), on each cycle of $G$ to form the vertex set $X$. It is possible that $x_i=x_j$ for $1\leq i, j\leq c(G)$. Suppose $G-X=T_1\cup T_2\cup\cdots\cup T_s\cup I$, where $I$ induces a null graph (i.e., no edges), $T_i~(1\leq i\leq s)$ is a tree with order at least $2$. By applying Lemma \ref{interlacing} and Theorem \ref{upper-bound}, we obtain
\begin{equation}\label{equ-3-1}
\begin{array}{lll}
m_\lambda(G)&\leq&m_\lambda(G-X)+|X|\\
&=&\sum_{i=1}^s m_\lambda(T_i)+|X|\\
&\leq&\sum_{i=1}^s\beta'(T_i)+c(G)\\
&=&\beta'(G)+c(G),
\end{array}
\end{equation}
as required. Next, we will characterize all connected graphs with $m_\lambda(G)=\beta'(G)+c(G)$. By Lemma \ref{m(T)=b(T)}, we only need to consider the case $X\neq\emptyset$. According to the proof of Equ.(\ref{equ-3-1}), equality holds if and only if each inequality involved to become an equality. Hence, $c(G-x_i)=c(G)-1~(1\leq i\leq c(G))$ for any $M^*$-unsaturated vertex on cycle, $m_\lambda(G-Z)=\beta'(G-Z)+c(G-Z')$ and $\beta'(G-Z)=\beta'(G)$ for any $Z\subseteq X$. Moreover, $m_\lambda(T_i)=\beta'(T_i)=1$ by Lemma \ref{m(T)=b(T)}. Suppose $z\in X$ and $H=G-(X-z)$, then $H$ is the union of an uncicyclic graph and some trees and isolated vertices. Moreover, $m_\lambda(H)=\beta'(H)+1=\beta'(G)+1$. We first claim $z$ is only connected to one of non-trivial components of $G-X$. Otherwise, assume $T_2+z$ is unicyclic graph and $T_1$ is adjacent to $z$. Since $m_\lambda(T_1-v)=m_\lambda(T_1)-1$ for any $v\in V(T_1)$ by Lemma \ref{star}, according to Lemma \ref{GuvH}, we have
\begin{equation*}
\begin{array}{lll}
m_\lambda(H)&=&m_\lambda(H-T_1-z)+m_\lambda(T_1)-1\\
&=&\sum_{i=2}^sm_\lambda(T_i)\\
&=&\beta'(H)-1,
\end{array}
\end{equation*}
a contradiction. Let the connected component of $H$ containing $T_2$ be $C^*$. Since the diameter for each $T_i~(1\leq i\leq s)$ is at most $3$ by Lemma \ref{m(T)=b(T)}, which implies that the  graph $C^*$ can take on any of the forms shown in Fig. \ref{fig-1}. Note that $m_\lambda(C^*)=\beta'(C^*)+1$. If  $C^*$ is in form $(1)$, then $m_\lambda(C^*)=2$. In this case, by applying Lemma \ref{interlacing}, we can conclude that $\lambda\in \sigma(T_2)$, $\lambda\in \sigma (C^*-v)$, which means that $a=c$. Moreover, $\lambda\in \sigma(C^*-w)$. It is worth mentioning that $\sigma(C^*-w)=\{\frac{1\pm\sqrt{1+4a}}{2},0^{n-4}, \frac{-1\pm\sqrt{1+4a}}{2}\}$ and $\sigma(T_2)=\sigma(C^*-v)=\{\pm\sqrt{a+1}, 0^{n-2}\}$. Through simple calculation, it is evident that $a=0$ for $\lambda\neq0$. Hence, $G\cong C_3$.

If $C^*$ has the form $(2)$ and $a, c\neq0$. In this case, $m_\lambda(C^*)=3$ and we can construct a non-zero eigenvector $Y$ of $A(C^*)$ corresponding to $\lambda$ such that $Y(z)=Y(v)=0$. However, by eigen-equation $A(C^*)Y=\lambda Y$, we have $Y=\textbf{0}$, a contradiction. Thus, at least one of $a, c$ must be $0$. Without loss of generality, say $c=0$. By applying Lemma \ref{interlacing}, we can show that $\lambda\in \sigma(T_2)$, and $\lambda\in \sigma (C^*-w)$. As a result, we may derive that $a=0$ and consequently $C^*\cong C_4$. However, since $\lambda\neq0$, it is impossible for $G\cong C_4$.

If $C^*$ has form $(3)$, by applying Lemma \ref{interlacing}, $\lambda\in \sigma(T_2)$, $\lambda\in \sigma (C^*-v)$ and $\lambda\in \sigma (C^*-w)$, which means $a=b=c$, as required.

If $C^*$ takes the form $(4)$ and at least one of $a, b, c\neq 0$, then $m_\lambda(C^*)=3$ and it is possible to construct an eigenvector $Y\neq \textbf{0}$ of $A(C^*)$ corresponding to $\lambda$ such that the components $Y(z)=Y(w)=0$. From the eigen-equation $A(C^*)Y=\lambda Y$, we have $Y=\textbf{0}$, which leads to a contradiction. Thus, we conclude that $a=b=c=0$ and $C^*\cong C_5$.

If $C^*$ has the form $(5)$ or $(7)$, then we only need to consider the case when $b\neq0$. In this case, $m_\lambda(C^*)=3$. Similar to the previous case in form $(4)$,  we can construct an eigenvector $Y\neq \textbf{0}$ of $A(C^*)$ corresponding to $\lambda$ such that $Y(z)=Y(w)=0$. Using the eigen-equation $A(C^*)Y=\lambda Y$, we arrive at the contradiction that $Y=\textbf{0}$.

If $C^*$ takes the form $(6)$, we need to consider only the case where $b\neq0$. If $a,c\neq0$, as similar as form $(2)$, by constructing an eigenvector $Y\neq \textbf{0}$ of $A(C^*)$ corresponding to $\lambda$ such that the components $Y(z)=Y(x)=0$,  we have $Y=\textbf{0}$, a contradiction. Thus, we only consider the case $c=0$. Under these conditions, $m_\lambda(C^*)=2$. We can construct an eigenvector $Y\neq \textbf{0}$ of $A(C^*)$ corresponding to $\lambda$ such that the components $Y(w)=0$. Using the eigen-equation $A(C^*)Y=\lambda Y$, we find that $Y(x)=0$, and subsequently $Y(z)=0$. This leads to a contradiction as $Y=\textbf{0}$. To sum up, $C^*\cong C_3(a, a,a)$ or $C^*\cong C_5$.

By adding vertices from $X$ to $G-X$ to form the vertex set $Z$, we continue this process until we find a vertex, say $y\in X-Z$, that is connected at least two vertices to a component of $G+Z$ which is not isomorphic to a tree. As discussed previously, each component of $G-(X-Z)$ is either a tree with diameter of at most $3$, or it is isomorphic to $C_3(a, a, a)$ or $C_5$. Let $H'=G-(X-Z)=H_1\cup H_2\cup\dots \cup H_t\cup I_1$ and $H^*=G-(X-Z-y)$, where $I_1$ induces a null graph, $H_i~(1\leq i\leq t)$ is a connected component with order at least $2$. We claim that $y$ is connected to only one  non-trivial component of $H'$ in $H^*$. Otherwise, assume that $y$ is adjacent to both $H_1$ and $H_2$, where $H_2$ is not a tree and $c(H_2+y)=c(H_2)+1\geq2$. Since $m_\lambda(H_1)=m_\lambda(H_1-v)+1$ for any $v\in V(H_1)$ by Lemmas \ref{star} and \ref{C3}, according to Lemma \ref{GuvH},
\begin{equation*}
\begin{array}{lll}
m_\lambda(H^*)&=&m_\lambda(H'-H_1)+m_\lambda(H_1)-1\\
&=&\sum_{i=1}^tm_\lambda(H_i)-1\\
&=&\sum_{i=1}^t (\beta'(H_i)+c(H_i))-1\\
&\leq&\beta'(H^*)+c(H^*)-2.
\end{array}
\end{equation*}
Since $m_\lambda(H^*)=\beta'(H^*)+c(H^*)$, this leads to a contradiction. Let the connected component of $H^*$ containing $H_2$ be $B^*$. Note that $H_2$ is isomorphic to $C_3(a,a, a)$ or $C_5$, and $m_\lambda(B^*)=\beta'(B^*)+c(B^*)\geq 3$. By Lemma \ref{C5}, a contradiction. As $G$ is connected, we conclude that $G\cong C_3(a,a, a)$ or $G\cong C_5$, as desired.

For the sufficiency part, by simple calculation, we have $m_{-\frac{\sqrt{5}+1}{2}}(C_5)=m_{\frac{\sqrt{5}-1}{2}}(C_5)=2$, $m_3(C_3)=2$ and $m_{-\frac{\sqrt{1+4a}+1}{2}}(C_3(a,a,a))=m_{\frac{\sqrt{1+4a}-1}{2}}(C_3(a,a,a))=2~(a\neq0)$, the proof is complete.
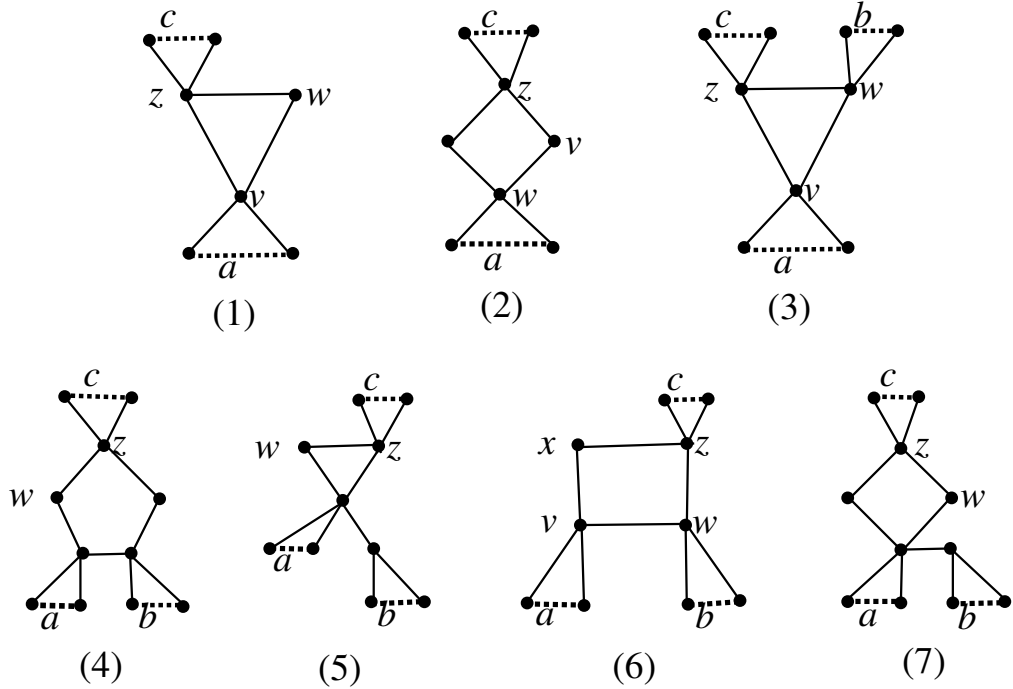
\begin{figure}[htbp]
\centering
\begin{tikzpicture}[x=1.00mm, y=1.00mm, inner xsep=0pt, inner ysep=0pt, outer xsep=0pt, outer ysep=0pt,scale=0.7]
\path[line width=0mm] (49.42,2.37) rectangle +(194.14,134.08);
\definecolor{L}{rgb}{0,0,0}
\definecolor{F}{rgb}{0,0,0}
\path[line width=0.30mm, draw=L, fill=F] (85.12,116.54) circle (1.00mm);
\path[line width=0.30mm, draw=L, fill=F] (105.80,116.54) circle (1.00mm);
\path[line width=0.30mm, draw=L, fill=F] (95.46,97.26) circle (1.00mm);
\path[line width=0.30mm, draw=L, fill=F] (145.58,118.60) circle (1.00mm);
\path[line width=0.30mm, draw=L, fill=F] (134.77,107.78) circle (1.00mm);
\path[line width=0.30mm, draw=L, fill=F] (144.64,97.68) circle (1.00mm);
\path[line width=0.30mm, draw=L, fill=F] (154.99,107.78) circle (1.00mm);
\path[line width=0.30mm, draw=L, fill=F] (85.59,86.45) circle (1.00mm);
\path[line width=0.30mm, draw=L, fill=F] (105.33,86.45) circle (1.00mm);
\path[line width=0.30mm, draw=L, fill=F] (135.48,88.04) circle (1.00mm);
\path[line width=0.30mm, draw=L, fill=F] (154.75,87.57) circle (1.00mm);
\path[line width=0.30mm, draw=L] (84.88,116.30) -- (95.46,96.79);
\path[line width=0.30mm, draw=L] (105.80,116.30) -- (95.93,97.03);
\path[line width=0.30mm, draw=L] (84.65,116.54) -- (106.50,116.77);
\path[line width=0.30mm, draw=L] (95.46,97.26) -- (85.59,86.69);
\path[line width=0.30mm, draw=L] (95.46,97.03) -- (105.09,85.98);
\path[line width=0.30mm, draw=L] (146.05,118.56) -- (135.24,107.31);
\path[line width=0.30mm, draw=L] (134.30,107.55) -- (144.64,97.44);
\path[line width=0.30mm, draw=L] (155.22,107.78) -- (144.88,97.21);
\path[line width=0.30mm, draw=L] (144.17,97.44);
\path[line width=0.30mm, draw=L] (145.11,97.91) -- (135.24,87.34);
\path[line width=0.30mm, draw=L] (144.41,97.68) -- (154.99,87.81);
\path[line width=0.30mm, draw=L] (145.35,118.56) -- (155.69,106.81);
\path[line width=0.30mm, draw=L, fill=F] (69.45,49.91) circle (1.00mm);
\path[line width=0.30mm, draw=L, fill=F] (60.52,40.04) circle (1.00mm);
\path[line width=0.30mm, draw=L, fill=F] (80.03,39.80) circle (1.00mm);
\path[line width=0.30mm, draw=L, fill=F] (65.46,29.46) circle (1.00mm);
\path[line width=0.30mm, draw=L, fill=F] (74.62,29.46) circle (1.00mm);
\path[line width=0.30mm, draw=L] (69.45,50.11) -- (60.29,39.77);
\path[line width=0.30mm, draw=L] (60.29,40.00) -- (65.22,29.43);
\path[line width=0.30mm, draw=L] (69.22,50.58) -- (80.26,39.53);
\path[line width=0.30mm, draw=L] (80.03,40.00) -- (74.62,29.19);
\path[line width=0.30mm, draw=L] (65.46,29.16) -- (75.33,29.39);
\path[line width=0.30mm, draw=L, fill=F] (55.82,19.79) circle (1.00mm);
\path[line width=0.30mm, draw=L, fill=F] (64.99,19.55) circle (1.00mm);
\path[line width=0.30mm, draw=L, fill=F] (74.86,19.79) circle (1.00mm);
\path[line width=0.30mm, draw=L, fill=F] (84.49,19.32) circle (1.00mm);
\path[line width=0.30mm, draw=L] (65.46,29.43) -- (55.35,19.55);
\path[line width=0.30mm, draw=L] (64.99,29.90) -- (64.99,19.55);
\path[line width=0.30mm, draw=L] (74.39,29.90) -- (74.86,19.55);
\path[line width=0.30mm, draw=L] (74.62,28.96) -- (84.96,19.32);
\path[line width=0.30mm, draw=L, fill=F] (107.53,49.60) circle (1.00mm);
\path[line width=0.30mm, draw=L, fill=F] (121.63,49.84) circle (1.00mm);
\path[line width=0.30mm, draw=L, fill=F] (114.81,39.50) circle (1.00mm);
\path[line width=0.30mm, draw=L, fill=F] (120.69,30.33) circle (1.00mm);
\path[line width=0.30mm, draw=L, fill=F] (100.95,30.33) circle (1.00mm);
\path[line width=0.30mm, draw=L, fill=F] (109.17,30.33) circle (1.00mm);
\path[line width=0.30mm, draw=L, fill=F] (120.22,20.22) circle (1.00mm);
\path[line width=0.30mm, draw=L, fill=F] (130.32,20.22) circle (1.00mm);
\path[line width=0.30mm, draw=L] (106.35,49.60) -- (122.10,50.07);
\path[line width=0.30mm, draw=L] (107.76,49.60) -- (115.05,39.03);
\path[line width=0.30mm, draw=L] (122.10,49.84) -- (115.05,39.26);
\path[line width=0.30mm, draw=L] (115.05,39.26) -- (100.71,30.10);
\path[line width=0.30mm, draw=L] (114.81,39.26) -- (109.17,30.10);
\path[line width=0.30mm, draw=L] (114.34,39.50) -- (120.69,30.10);
\path[line width=0.30mm, draw=L] (120.45,30.33) -- (130.32,19.75);
\path[line width=0.30mm, draw=L] (120.69,30.57) -- (120.69,20.46);
\path[line width=0.30mm, draw=L, fill=F] (159.47,50.07) circle (1.00mm);
\path[line width=0.30mm, draw=L, fill=F] (180.15,50.31) circle (1.00mm);
\path[line width=0.30mm, draw=L, fill=F] (159.94,34.80) circle (1.00mm);
\path[line width=0.30mm, draw=L, fill=F] (179.92,34.80) circle (1.00mm);
\path[line width=0.30mm, draw=L, fill=F] (149.83,19.99) circle (1.00mm);
\path[line width=0.30mm, draw=L, fill=F] (160.64,19.52) circle (1.00mm);
\path[line width=0.30mm, draw=L, fill=F] (180.39,19.75) circle (1.00mm);
\path[line width=0.30mm, draw=L, fill=F] (190.26,20.46) circle (1.00mm);
\path[line width=0.30mm, draw=L] (159.00,49.60) -- (181.56,50.31);
\path[line width=0.30mm, draw=L] (159.23,50.31) -- (159.94,34.56);
\path[line width=0.30mm, draw=L] (180.39,50.31) -- (180.15,34.09);
\path[line width=0.30mm, draw=L] (159.70,34.80) -- (180.39,35.03);
\path[line width=0.30mm, draw=L] (159.94,34.80) -- (150.07,20.46);
\path[line width=0.30mm, draw=L] (160.17,34.80) -- (160.64,19.52);
\path[line width=0.30mm, draw=L] (179.92,35.27) -- (180.15,19.75);
\path[line width=0.30mm, draw=L] (179.68,35.50) -- (190.26,20.46);
\path[line width=0.30mm, draw=L, fill=F] (220.81,49.37) circle (1.00mm);
\path[line width=0.30mm, draw=L, fill=F] (210.71,39.97) circle (1.00mm);
\path[line width=0.30mm, draw=L, fill=F] (230.45,39.97) circle (1.00mm);
\path[line width=0.30mm, draw=L, fill=F] (220.81,30.10) circle (1.00mm);
\path[line width=0.30mm, draw=L] (220.81,49.60) -- (210.71,39.50);
\path[line width=0.30mm, draw=L] (210.71,40.20) -- (221.28,29.63);
\path[line width=0.30mm, draw=L] (221.05,49.37) -- (231.39,39.26);
\path[line width=0.30mm, draw=L] (230.45,40.20) -- (220.58,29.39);
\path[line width=0.30mm, draw=L, fill=F] (210.71,20.26) circle (1.00mm);
\path[line width=0.30mm, draw=L, fill=F] (220.81,20.02) circle (1.00mm);
\path[line width=0.30mm, draw=L, fill=F] (230.21,30.37) circle (1.00mm);
\path[line width=0.30mm, draw=L, fill=F] (230.68,20.02) circle (1.00mm);
\path[line width=0.30mm, draw=L, fill=F] (240.55,20.26) circle (1.00mm);
\path[line width=0.30mm, draw=L] (221.05,30.13) -- (210.71,20.73);
\path[line width=0.30mm, draw=L] (221.05,30.37) -- (220.81,20.02);
\path[line width=0.30mm, draw=L] (220.58,30.13) -- (230.92,30.37);
\path[line width=0.30mm, draw=L] (230.68,30.37) -- (230.68,20.26);
\path[line width=0.30mm, draw=L] (230.21,30.37) -- (240.79,20.49);
\path[line width=0.75mm, draw=L, dash pattern=on 0.75mm off 0.50mm] (55.58,19.59) -- (65.22,19.35);
\path[line width=0.60mm, draw=L, dash pattern=on 0.60mm off 0.50mm] (75.09,19.59) -- (84.73,19.59);
\path[line width=0.60mm, draw=L, dash pattern=on 0.60mm off 0.50mm] (85.82,85.95) -- (105.33,86.18);
\path[line width=0.69mm, draw=L, dash pattern=on 0.69mm off 0.50mm] (135.48,88.24) -- (154.99,88.24);
\path[line width=0.69mm, draw=L, dash pattern=on 0.69mm off 0.50mm] (101.18,30.40) -- (109.64,30.17);
\path[line width=0.69mm, draw=L, dash pattern=on 0.69mm off 0.50mm] (120.22,19.82) -- (130.79,20.29);
\path[line width=0.69mm, draw=L, dash pattern=on 0.69mm off 0.50mm] (149.60,19.82) -- (160.64,19.59);
\path[line width=0.69mm, draw=L, dash pattern=on 0.69mm off 0.50mm] (180.62,19.59) -- (190.96,20.29);
\path[line width=0.69mm, draw=L, dash pattern=on 0.69mm off 0.50mm] (210.47,20.29) -- (221.05,20.29);
\path[line width=0.69mm, draw=L, dash pattern=on 0.69mm off 0.50mm] (230.45,19.82) -- (240.79,20.06);
\draw(90.05,72.82) node[anchor=base west]{\fontsize{14.23}{17.07}\selectfont $(1)$};
\draw(140.88,74.41) node[anchor=base west]{\fontsize{14.23}{17.07}\selectfont $(2)$};
\draw(91.23,81.68) node[anchor=base west]{\fontsize{14.23}{17.07}\selectfont $a$};
\draw(141.59,83.27) node[anchor=base west]{\fontsize{14.23}{17.07}\selectfont $a$};
\draw(64.75,5.92) node[anchor=base west]{\fontsize{14.23}{17.07}\selectfont $(4)$};
\draw(110.11,5.45) node[anchor=base west]{\fontsize{14.23}{17.07}\selectfont $(5)$};
\draw(166.05,5.69) node[anchor=base west]{\fontsize{14.23}{17.07}\selectfont $(6)$};
\draw(221.05,6.63) node[anchor=base west]{\fontsize{14.23}{17.07}\selectfont $(7)$};
\draw(57.47,15.09) node[anchor=base west]{\fontsize{14.23}{17.07}\selectfont $a$};
\draw(76.03,15.32) node[anchor=base west]{\fontsize{14.23}{17.07}\selectfont $b$};
\draw(101.65,26.14) node[anchor=base west]{\fontsize{14.23}{17.07}\selectfont $a$};
\draw(121.39,15.56) node[anchor=base west]{\fontsize{14.23}{17.07}\selectfont $b$};
\draw(151.48,15.79) node[anchor=base west]{\fontsize{14.23}{17.07}\selectfont $a$};
\draw(181.80,15.32) node[anchor=base west]{\fontsize{14.23}{17.07}\selectfont $b$};
\draw(212.82,15.56) node[anchor=base west]{\fontsize{14.23}{17.07}\selectfont $a$};
\draw(231.62,15.09) node[anchor=base west]{\fontsize{14.23}{17.07}\selectfont $b$};
\draw(96.88,95.25) node[anchor=base west]{\fontsize{14.23}{17.07}\selectfont $v$};
\draw(77.99,114.59) node[anchor=base west]{\fontsize{14.23}{17.07}\selectfont $z$};
\draw(147.95,115.93) node[anchor=base west]{\fontsize{14.23}{17.07}\selectfont $z$};
\draw(147.06,95.48) node[anchor=base west]{\fontsize{14.23}{17.07}\selectfont $w$};
\draw(70.98,48.49) node[anchor=base west]{\fontsize{14.23}{17.07}\selectfont $z$};
\draw(123.21,47.60) node[anchor=base west]{\fontsize{14.23}{17.07}\selectfont $z$};
\draw(181.67,48.49) node[anchor=base west]{\fontsize{14.23}{17.07}\selectfont $z$};
\draw(223.45,48.05) node[anchor=base west]{\fontsize{14.23}{17.07}\selectfont $z$};
\draw(51.42,38.49) node[anchor=base west]{\fontsize{14.23}{17.07}\selectfont $w$};
\draw(98.09,47.60) node[anchor=base west]{\fontsize{14.23}{17.07}\selectfont $w$};
\draw(232.12,38.27) node[anchor=base west]{\fontsize{14.23}{17.07}\selectfont $w$};
\draw(181.22,33.38) node[anchor=base west]{\fontsize{14.23}{17.07}\selectfont $w$};
\draw(152.38,33.40) node[anchor=base west]{\fontsize{14.23}{17.07}\selectfont $v$};
\draw(151.71,48.37) node[anchor=base west]{\fontsize{14.23}{17.07}\selectfont $x$};
\path[line width=0.30mm, draw=L, fill=F] (78.09,126.95) circle (1.00mm);
\path[line width=0.30mm, draw=L, fill=F] (90.60,127.18) circle (1.00mm);
\path[line width=0.30mm, draw=L] (77.42,126.95) -- (85.46,116.45);
\path[line width=0.30mm, draw=L] (90.82,127.18) -- (84.79,116.01);
\path[line width=0.30mm, draw=L, fill=F] (137.96,128.29) circle (1.00mm);
\path[line width=0.30mm, draw=L, fill=F] (150.92,128.74) circle (1.00mm);
\path[line width=0.30mm, draw=L] (138.18,128.07) -- (145.56,118.69);
\path[line width=0.30mm, draw=L] (150.47,129.19) -- (146.45,118.46);
\path[line width=0.30mm, draw=L, fill=F] (62.00,59.23) circle (1.00mm);
\path[line width=0.30mm, draw=L, fill=F] (74.74,59.01) circle (1.00mm);
\path[line width=0.30mm, draw=L, fill=F] (117.86,58.56) circle (1.00mm);
\path[line width=0.30mm, draw=L, fill=F] (127.01,58.78) circle (1.00mm);
\path[line width=0.30mm, draw=L, fill=F] (175.94,58.56) circle (1.00mm);
\path[line width=0.30mm, draw=L, fill=F] (184.21,58.78) circle (1.00mm);
\path[line width=0.30mm, draw=L, fill=F] (215.71,58.97) circle (1.00mm);
\path[line width=0.30mm, draw=L, fill=F] (224.20,59.20) circle (1.00mm);
\path[line width=0.30mm, draw=L] (61.78,59.45) -- (69.60,49.85);
\path[line width=0.30mm, draw=L] (74.51,59.68) -- (69.38,49.40);
\path[line width=0.30mm, draw=L] (117.86,58.78) -- (122.10,48.51);
\path[line width=0.30mm, draw=L] (127.24,59.23) -- (121.21,48.51);
\path[line width=0.30mm, draw=L] (175.94,58.78) -- (180.63,50.07);
\path[line width=0.30mm, draw=L] (184.65,59.23) -- (179.96,49.62);
\path[line width=0.30mm, draw=L] (215.48,58.78) -- (221.07,48.95);
\path[line width=0.30mm, draw=L] (224.20,59.01) -- (220.62,48.73);
\path[line width=0.30mm, draw=L, fill=F] (190.56,117.65) circle (1.00mm);
\path[line width=0.30mm, draw=L, fill=F] (211.25,117.65) circle (1.00mm);
\path[line width=0.30mm, draw=L, fill=F] (200.90,98.38) circle (1.00mm);
\path[line width=0.30mm, draw=L, fill=F] (191.03,87.57) circle (1.00mm);
\path[line width=0.30mm, draw=L, fill=F] (210.78,87.57) circle (1.00mm);
\path[line width=0.30mm, draw=L] (190.33,117.42) -- (200.90,97.91);
\path[line width=0.30mm, draw=L] (211.25,117.42) -- (201.38,98.15);
\path[line width=0.30mm, draw=L] (190.09,117.65) -- (211.95,117.89);
\path[line width=0.30mm, draw=L] (200.90,98.38) -- (191.03,87.81);
\path[line width=0.30mm, draw=L] (200.90,98.15) -- (210.54,87.10);
\path[line width=0.60mm, draw=L, dash pattern=on 0.60mm off 0.50mm] (191.27,87.07) -- (210.78,87.30);
\draw(195.50,73.94) node[anchor=base west]{\fontsize{14.23}{17.07}\selectfont $(3)$};
\draw(196.67,82.80) node[anchor=base west]{\fontsize{14.23}{17.07}\selectfont $a$};
\draw(202.32,96.37) node[anchor=base west]{\fontsize{14.23}{17.07}\selectfont $v$};
\draw(183.43,115.71) node[anchor=base west]{\fontsize{14.23}{17.07}\selectfont $z$};
\path[line width=0.30mm, draw=L, fill=F] (183.54,128.07) circle (1.00mm);
\path[line width=0.30mm, draw=L, fill=F] (196.05,128.29) circle (1.00mm);
\path[line width=0.30mm, draw=L] (182.87,128.07) -- (190.91,117.57);
\path[line width=0.30mm, draw=L] (196.27,128.29) -- (190.24,117.12);
\path[line width=0.30mm, draw=L, fill=F] (210.35,128.55) circle (1.00mm);
\path[line width=0.30mm, draw=L, fill=F] (220.18,128.77) circle (1.00mm);
\path[line width=0.30mm, draw=L] (210.35,128.10) -- (211.24,117.61);
\path[line width=0.30mm, draw=L] (220.62,129.22) -- (211.01,117.16);
\path[line width=0.60mm, draw=L, dash pattern=on 0.60mm off 0.50mm] (78.09,127.18) -- (90.60,127.40);
\path[line width=0.60mm, draw=L, dash pattern=on 0.60mm off 0.50mm] (138.18,128.29) -- (150.92,128.52);
\path[line width=0.60mm, draw=L, dash pattern=on 0.60mm off 0.50mm] (183.98,127.85) -- (196.27,127.85);
\path[line width=0.60mm, draw=L, dash pattern=on 0.60mm off 0.50mm] (210.35,128.74) -- (220.62,128.74);
\path[line width=0.60mm, draw=L, dash pattern=on 0.60mm off 0.50mm] (61.56,59.45) -- (74.96,59.01);
\path[line width=0.60mm, draw=L, dash pattern=on 0.60mm off 0.50mm] (117.63,58.78) -- (127.68,58.78);
\path[line width=0.60mm, draw=L, dash pattern=on 0.60mm off 0.50mm] (176.16,58.56) -- (184.65,58.78);
\path[line width=0.60mm, draw=L, dash pattern=on 0.60mm off 0.50mm] (215.71,59.23) -- (224.64,59.01);
\draw(80.10,129.41) node[anchor=base west]{\fontsize{14.23}{17.07}\selectfont $c$};
\draw(141.09,129.41) node[anchor=base west]{\fontsize{14.23}{17.07}\selectfont $c$};
\draw(185.55,129.41) node[anchor=base west]{\fontsize{14.23}{17.07}\selectfont $c$};
\draw(211.79,129.413) node[anchor=base west]{\fontsize{14.23}{17.07}\selectfont $b$};
\draw(65.58,61.17) node[anchor=base west]{\fontsize{14.23}{17.07}\selectfont $c$};
\draw(118.75,60.50) node[anchor=base west]{\fontsize{14.23}{17.07}\selectfont $c$};
\draw(176.39,60.50) node[anchor=base west]{\fontsize{14.23}{17.07}\selectfont $c$};
\draw(216.82,61.17) node[anchor=base west]{\fontsize{14.23}{17.07}\selectfont $c$};
\draw(107.88,114.74) node[anchor=base west]{\fontsize{14.23}{17.07}\selectfont $w$};
\draw(212.79,115.88) node[anchor=base west]{\fontsize{14.23}{17.07}\selectfont $w$};
\draw(157.03,105.33) node[anchor=base west]{\fontsize{14.23}{17.07}\selectfont $v$};
\end{tikzpicture}%
\caption{ All possible forms of graph $C^*$.}\label{fig-1}
\end{figure}
\end{proof}
Note that $\sigma(K_{1,n-1})=\{-\sqrt{n-1},0^{n-2},\sqrt{n-1}\}$, $\sigma(C_3(a,a,a))=\{(-\frac{\sqrt{1+4a}+1}{2})^2,1-\sqrt{1+a},(\frac{\sqrt{1+4a}-1}{2})^2,1+\sqrt{1+a}\}$,
$\sigma(C_5)=\{(-\frac{\sqrt{5}+1}{2})^2,(\frac{\sqrt{5}-1}{2})^2,2\}$. $K_{1,n-1}$ cannot have eigenvalues of $-\frac{\sqrt{5}+1}{2}$ and $\frac{\sqrt{5}-1}{2}$, and $C_3(a,a,a)$ has eigenvalues of $-\frac{\sqrt{5}+1}{2}$ and $\frac{\sqrt{5}-1}{2}$ if and only if $a=1$.

The following corollary can be derived from Theorem \ref{thm-upper-bound}.
\begin{cor}\label{m=b+c-1}
Let $G$ be a connected graph with $\lambda \neq 0$ as an eigenvalue of multiplicity $k \geq 1$. If the diameter of $G$ is at least $4$, then $m_\lambda(G) \leq \beta^{\prime}(G)+c(G)-1$.
\end{cor}
In the following, we will characterize all connected graphs with diameter at least $4$ and $\beta'(G)\geq3$ that have a nonzero eigenvalue $\lambda$ with multiplicity $\beta'(G)+c(G)-1$.

\begin{lem}\label{components}
Let $G$ be a connected graph with diameter at least $4$ and $\beta'(G)\geq3$ which has a nonzero eigenvalue $\lambda$. For any $M^*$-unsaturated vertex $x$ on the cycle, suppose $G-x=H_1\cup H_2\cup\cdots\cup H_l\cup I~(l\geq1)$, where $I$ induces a null graph and $H_i~(1\leq i\leq l)$ is a connected component with order at least $2$. If $m_\lambda(G)=\beta'(G)+c(G)-1$, then $m_\lambda(G-x)=\beta'(G-x)+c(G-x)-1$, $c(G-x)=c(G)-1$, and $l\leq2$. Furthermore, if $G-x=H_1\cup H_2\cup I$, then we have $m_\lambda(H_1)=\beta'(H_1)+c(H_1)-1$, $m_\lambda(H_2)=\beta'(H_2)+c(H_2)$, and $x$ is adjacent to two vertices of $H_2$.
\end{lem}
\begin{proof}
Let $G-x=H_1\cup H_2\cup\cdots\cup H_l\cup I~(l\geq1)$. According to Lemma \ref{interlacing}, we will discuss the following two cases.
{\flushleft  \bf Case 1.} If $m_\lambda(G)\leq m_\lambda(G-x)$.

In this case, we have
\begin{equation}
\begin{array}{lll}
\beta'(G)+c(G)-1=m_\lambda(G)&\leq&m_\lambda(G-x)\\
&=&\sum_{i=1}^lm_\lambda(H_i)\\
&\leq&\sum_{i=1}^l(\beta'(H_i)+c(H_i))\\
&\leq&\beta'(G)+c(G)-1.
\end{array}
\end{equation}
 The equality holds if and only if each inequality involved to become an equality. Hence, all components of $G-x$ with $m_\lambda (H_i)=\beta'(H_i)+c(H_i)$ for $1\leq i\leq l$, $c(G-x)=c(G)-1$ and $\beta'(G)=\beta'(G-x)$.

According to Theorem \ref{thm-upper-bound},  $\beta'(H_i)=1~(1\leq i\leq l)$. Since $\beta'(G)\geq 3$, we have $l\geq3$. Suppose that $x$ is adjacent two vertices of $H_1$, by applying Lemmas \ref{GuvH},\ref{star}, and \ref{C3}, we have
\begin{equation*}
\begin{array}{lll}
m_\lambda(G)&=&m_\lambda(G-H_l-x)+m_\lambda(H_l)-1\\
&=&\sum_{i=1}^{l-1}m_\lambda(H_i)+m_\lambda(H_l)-1\\
&=&\sum_{i=1}^{l}(\beta'(H_i)+c(H_i))-1\\
&=&\beta'(G)+c(G)-2,
\end{array}
\end{equation*}
a contradiction.
{\flushleft  \bf Case 2.} If $m_\lambda(G)=m_\lambda(G-x)+1$.

In this case,
\begin{equation*}
\begin{array}{lll}
\beta'(G)+c(G)-1=m_\lambda(G)&=&m_\lambda(G-x)+1\\
&=&\sum_{i=1}^lm_\lambda(H_i)+1\\
&\leq&\sum_{i=1}^l(\beta'(H_i)+c(H_i))+1\\
&\leq&\beta'(G)+c(G).
\end{array}
\end{equation*}
Thus, $m_\lambda(G)=\beta'(G)+c(G)-1$ hold either  $G-x$  has only one component (denoted by $H_1$) with $m_\lambda(H_1)=\beta'(H_1)+c(H_1)-1$, $c(G-x)=c(G)-1$, and other components (if they exit) with $m_\lambda(H_i)=\beta'(H_i)+c(H_i)$ for $2\leq i\leq l$, or $c(G-x)=c(G)-2$ and all components $G-x$ with $m_\lambda (H_i)=\beta'(H_i)+c(H_i)$ for $1\leq i\leq l$. If the former case where $l\geq3$ or the latter case occurs, similar to Case $1$, there exists a component, say $H_l$, with $m_\lambda(H_l)=\beta'(H_l)+c(H_l)$  that has only one vertex adjacent to $x$.  Applying Lemmas \ref{GuvH}, \ref{star}, and \ref{C3}, we obtain
\begin{equation}\label{equ-3-4-1}
\begin{array}{lll}
m_\lambda(G)&=&m_\lambda(G-H_l-x)+m_\lambda(H_l)-1\\
&=&\sum_{i=1}^{l-1}m_\lambda(H_i)+m_\lambda(H_l)-1\\
&=&\beta'(G)+c(G)-3,
\end{array}
\end{equation}
a contradiction. Thus, $l\leq2$ for the former case. If $x$ is adjacent to two vertices of $H_1$,  then we can reach a contradiction similar to  one given in Equ.(\ref{equ-3-4-1}). Therefore, $x$ is adjacent to two vertices of $H_2$.
\end{proof}
\begin{lem}\label{unciyclic-hold}
Let $G$ be a connected unicyclic graph with diameter at least $4$ and $\lambda$ as a nonzero eigenvalue. If $\beta'(G)\geq3$, then $m_\lambda(G)=\beta'(G)$ if and only if there exists a vertex $w$ such that $G-w=H_1 \cup \ldots \cup H_{s} \cup I~(s=\beta'(G))$, where $I$ induces a null graph, $H_1\cong C_3(a,a,a)$ with $m_\lambda(H_1)=2$, $H_i~(2\leq i\leq s)$ is a tree with diameter of at most $3$ with $m_\lambda(H_i)=1~$. Furthermore, either $diam(H_i+w)=2$ for some $i$, or $I=\emptyset$ and $diam(H_i+w)=3$ for all $i$.
\end{lem}
\begin{proof}
Firstly, we demonstrate the sufficiency part. We only need to show that $m_\lambda(G)=\beta'(G)=s$, where $s$ is the number of components in $G-w$. Consider a vertex $u_s$ which is adjacent to $w$ in $H_s$. Note that $m_\lambda(H_1)=2$, $m_\lambda(H_i)=1~(2\leq i\leq s)$, and each $H_i~(2\leq i\leq s)$ is a tree with diameter at most $3$. By applying Lemmas \ref{GuvH} and \ref{star},  $\lambda\notin \sigma(H_i-u_i)~(2\leq i\leq s)$, and
\begin{equation*}
m_\lambda(G)=m_\lambda(G-H_s-w)=\sum_{i=1}^{s-1} m_\lambda(H_i)=s,
\end{equation*}
which completes the proof that $m_\lambda(G)=s$. Note that either $diam(H_i+w)=2$ for some $i$, or $I=\emptyset$ and $diam(H_i+w)=3$ for all $i$. Consequently, $\beta'(G)=s$.

Next, we demonstrate the necessity part.  Note that there at least one $M^*$-unsaturated vertex on the cycle of $G$, which is denoted by $x$. Suppose $G-x=F_1\cup F_2\cup\cdots\cup F_l\cup I_1~(l\geq1)$, where $I_1$ induces a null graph and $F_i~(1\leq i\leq l)$ is a connected component with order at least $2$. By Lemma \ref{components}, $G-x$  has only one component (denoted by $F_1$) with $m_\lambda(F_1)=\beta'(F_1)-1$ and $l\leq2$.
According to Theorem \ref{extremal-tree}, $F_1$ can take on two different forms.

{\flushleft  \bf Case 1.}  $F_1$ has the form $(b)$ of Theorem \ref{extremal-tree}.

In this case, $m_\lambda(F_1)\geq2$. As $G$ is an unicyclic graph, there is only one component that has two vertices adjacent to $x$. Let $F_1-u=T_1\cup T_2\cup\cdots\cup T_{p+1}\cup I_2~(p\geq2)$, where $I_2$ induces a null graph and $T_i~(1\leq i\leq p+1)$ is a tree with order at least $2$.

If $l=1$ and the cycle of $G$ contains the vertex $u$, then since $x$ is adjacent to at most two components of $F_1-u$, there exists at least one component of $F_1-u$ that does not contain any vertices on the cycle of $G$, say $T_{p+1}$. If $x$ is adjacent to  one  component of $\{T_i|1\leq i\leq p\}$, without loss of generality, we set $T^*=T_1+x+I_1$. If $x$ is adjacent to two components of $\{T_i|1\leq i\leq p\}$,  we set $T^*=T_1+T_2+x+I_1$. Since $F_1$ has the form $(b)$ of Theorem \ref{extremal-tree} and $x$ is an $M^*$-unsaturated vertex, $\beta'(T^*)+\sum_{i=j}^{p+1}\beta'(T_i)=\beta'(G)($j=2$~\text{or}~3)$. Using Lemma \ref{GuvH} again, we have
\begin{equation*}
\begin{array}{lll}
m_\lambda(G)&=&m_\lambda(G-T_{p+1}-u )+m_\lambda(T_{p+1})-1\\
&=&m_\lambda(T^*)+\sum_{i=j}^pm_\lambda(T_i) \quad ($j=2$~\text{or}~3)\\
&\leq&\beta'(T^*)+\sum_{i=j}^p\beta'(T_i)\\
&=&\beta'(G)-1,
\end{array}
\end{equation*}
which leads to a contradiction. Thus $u$ is not on the cycle of $G$, which implies that $x$ is adjacent to  one  component of $\{T_i|1\leq i\leq p+1\}$, say $T_1$, and let the connected component of $G-u$ containing $T_1+x$ be $U_1^*$. Since $F_1$ has the form $(b)$ of Theorem \ref{extremal-tree} and $x$ is an $M^*$-unsaturated vertex, $\beta'(U_1^*)+\sum_{i=2}^{p+1}\beta'(T_i)=\beta'(G)$. According to Lemma \ref{GuvH}, we have
\begin{equation*}
\begin{array}{lll}
\beta'(G)=m_\lambda(G)&=&m_\lambda(G-T_{p+1}-u)+m_\lambda(T_{p+1})-1\\
&=&m_\lambda(U_1^*)+\sum_{i=2}^{p}m_\lambda(T_i)\\
&\leq&\beta'(U_1^*)+1+\sum_{i=2}^{p}\beta'(T_i)\\
&=&\beta'(G).
\end{array}
\end{equation*}
Therefore, the equality hold if and only if $m_\lambda(U_1^*)=\beta'(U_1^*)+1$ and $\beta'(G)=\beta'(U_1^*)+\sum_{i=2}^{p}\beta'(T_i)+1$. By Theorem \ref{thm-upper-bound},  $U_1^*$ is isomorphic to $C_3(a, a, a)$ or $C_5$. If $U_1^*\cong C_5$, then each $T_i~(2\leq i\leq p+1)$ is a tree with diameter $3$ and $\beta'(G)=\beta'(U_1^*)+\sum_{i=2}^{p}\beta'(T_i)+2$, which is a contradiction. Hence, $U_1^*\cong C_3(a,a,a)$, $\beta'(G)=s=p+1$, and $w=u$, as required.

If $l=2$, then $x$ is adjacent to two vertices of $F_2$ by Lemma \ref{components}. Let the connected component of $G-u$ containing $x+F_2$ be $F_2^*$ and $F^*=G-V(F_2^*)-V(T_{p+1})-u$. Since $F_1$ has the form $(b)$ of Theorem 1.2 and $F_2$ is a tree with diameter at most $3$, $\beta'(G)=\beta'(F_2^*)+\beta'(F^*)+\beta'(T_{p+1})$. Using Lemma \ref{GuvH} again, we have
\begin{equation*}
\begin{array}{lll}
m_\lambda(G)&=&m_\lambda(G-T_{p+1}-u)+m_\lambda(T_{p+1})-1\\
&=&m_\lambda(F_2^*)+m_\lambda(F^*)\\
&\leq&\beta'(F_2^*)+1+\beta'(F^*)\\
&=&\beta'(G),
\end{array}
\end{equation*}
the equality hold if and only if $m_\lambda(F_2^*)=\beta'(F_2^*)+1$ and $\beta'(G)=\beta'(F_2^*)+\beta'(F^*)+1$. By Theorem \ref{thm-upper-bound},  $F_2^*$ is isomorphic to $C_3(a, a, a)$ or $C_5$. If $F_2^* \cong C_5$, then each $T_i~(1\leq i\leq p+1)$ is a tree with diameter $3$ and $\beta'(G)=\beta'(F_2^*)+\beta'(F^*)+2$, which is a contradiction. Hence, $F_2^*\cong C_3(a,a,a)$, $F^*\cong\bigcup_{i=1}^{p}T_i$, and $w=u$, as required.

{\flushleft  \bf Case 2.}  $F_1$  has form $(a)$ of Theorem \ref{extremal-tree}.

In this case, $\beta'(F_1)=2$. Since $\beta'(G)\geq3$, we have $l=2$. Suppose  $F_2'=F_2+x+I_1$. Let $z$ be the only vertex in $F_1$ adjacent to $x$, and $G-z=F'_1\cup \cdots\cup F'_k\cup I_3$, where $I_3$ induces a null graph and $F'_i~(1\leq i\leq k)$ is a connected component with order at least $2$. Note that $F_2$ is a tree with a diameter of at most $3$. Therefore, if $F_2\ncong P_2$ or $\beta'(G)=\beta'(G-z)$, then we can find another induced matching $M^{*'}$ such that there is an unsaturated vertex $y$ in the cycle of $F_2'$ and $G-y$ has non-trivial component, say $K_1$, with $\beta'(K_1)\geq3$. Therefore, $K_1$ is of the form $(b)$ mentioned in Theorem \ref{extremal-tree}. We can proceed similar as to Case $1$ and obtain the desired result. Thus, we only need to consider the case $F_2\cong P_2$ and $\beta'(G)=\beta'(G-z)+1$. Since $\beta'(G)=\beta'(G-z)+1$ and $x$ is an $M^*$-unsaturated vertex, $\beta'(F_2)=\beta'(F'_2)$ and $\sum_{i=1}^k\beta'(F_i')=\beta'(G)-1$. If $m_\lambda(F_2')=m_\lambda(F_2)+1$, by applying Lemma \ref{GuvH}, we have
\begin{equation*}
\begin{array}{lll}
m_\lambda(G)&=&m_\lambda(G-F_2'-z)+m_\lambda(F_2')-1\\
&\leq&\sum_{i=1, i\neq2}^k\beta'(F_i')+\beta'(F_2')+1-1\\
&=&\beta'(G)-1,
\end{array}
\end{equation*}
a contradiction. Next, we only consider the case $m_\lambda(F_2')\leq m_\lambda(F_2)=\beta'(F_2)=\beta'(F_2')$. According to Lemma \ref{interlacing}, we have
\begin{equation*}
\begin{array}{lll}
m_\lambda(G)&\leq&m_\lambda(G-z)+1\\
&=&\sum_{i=1, i\neq2}^km_\lambda(F_i')+m_\lambda(F_2')+1\\
&\leq&\sum_{i=1, i\neq2}^k\beta'(F_i')+\beta'(F'_2)+1\\
&=&\beta'(G).
\end{array}
\end{equation*}
Thus, $m_\lambda(G)=\beta'(G)$ if and only if each inequality involved to become an equality. That is, $m_\lambda(G)=m_\lambda(G-z)+1$, each components of $G-z$ with $m_\lambda(F_i')=\beta(F'_i)~(1\leq i\leq k)$, and $\beta'(G)=\beta'(G-z)+1$. Since $k\geq2$, there at least one components of $G-z$, say $F'_k~(k\neq2)$, such that $m_\lambda(F'_k)=\beta'(F'_k)$ and has only one vertex adjacent to $z$. By Lemma \ref{GuvH}, we have
\begin{equation}\label{equ-3-5-1}
\begin{array}{lll}
m_\lambda(G)&=&m_\lambda(G-z-F_k')+m_\lambda(F_k')-1\\
&=&\sum_{i=1}^{k}m_\lambda(F_k')-1\\
&=&\beta'(G)-2,
\end{array}
\end{equation}
a contradiction.
\end{proof}
\begin{thm}\label{extremal-graph}
Let $G$ be a connected  graph with a diameter of at least $4$ and $\lambda$ as a nonzero eigenvalue. If $\beta'(G)\geq3$, then $m_\lambda(G)=\beta'(G)+c(G)-1$ if and only if there exists a vertex $w$ such that $G-w=H_1 \cup \ldots \cup H_{c(G)} \ldots\cup H_s\cup I$, where $s=\beta'(G) \geq 3$, $I$ induces a null graph (without edges), each $H_i~(1\leq i\leq c(G))$ is isomorphic to $C_3(a,a,a)$ with $m_\lambda(H_i)=2$, and $H_j~(c(G)+1\leq j\leq s)$ is a tree with diameter of at most $3$ and $m_\lambda(H_j)=1$. Furthermore, either $diam(H_i+w)=2$ for some $i$, or $I=\emptyset$ and $diam(H_i+w)=3$ for all $i~(1\leq i\leq s)$.
\end{thm}
\begin{proof}
we demonstrate the sufficiency part by showing that $m_\lambda(G)=c(G)+\beta'(G)-1$. Note that $s\geq3$ and $m_\lambda(H_i)=m_\lambda(H_i-u_i)+1~(1\leq i\leq s)$ by lemmas \ref{star} and \ref{C3}, where $u_i$ is a vertex in $H_i$ that is adjacent to $w$. According to Lemma \ref{GuvH} and Theorem \ref{thm-upper-bound}, we obtain
\begin{equation*}
m_\lambda(G)=m_\lambda(G-H_s-w)+m_\lambda(H_s)-1=\sum_{i=1}^{s} m_\lambda(H_i)-1=\beta'(G)+c(G)-1.
\end{equation*}

Next, we demonstrate the necessity part by induction on $c(G)$.  If $c(G)=1$,  Lemma \ref{unciyclic-hold} implies the result, and so we proceed to the case where $c(G)\geq2$.  Note that there at least one $M^*$-unsaturated vertex on any cycle of $G$, and we choose one (called $x$) and set $G-x=H'_1\cup\cdots\cup H'_l\cup I_1~(l\leq2)$ by Lemma \ref{components}. Moreover, $G-x$ must have only one component (which we denote as $H'_1$) with $m_\lambda(H'_1)=\beta'(H'_1)+c(H'_1)-1$. If $l=2$, then $m_\lambda(H'_2)=\beta'(H'_2)+c(H'_2)$ and $H'_2$ has two vertices adjacent to $x$.

{\flushleft  \bf Case 1.} $l=1$.

In this case, $\beta'(G)=\beta'(H'_1)=s\geq3$. By induction hypothesis, suppose $H'_1-u=K_1 \cup \ldots \cup K_s\cup I_2~(s\geq3)$. Next, we assert that the cycle of $G$ does not include the vertex $u$. If it did, then $x$ would be adjacent to at most two components of $H'_1-u$. Thus, at least one component of $H'_1-u$, denoted as $K_s$, would not contain any vertices on the cycle containing the vertex $u$ of $G$. If $u$ is adjacent to one  component of $\{K_i|1\leq i\leq s\}$, without loss of generality, we set $K^*=K_1+x+I_1$. If $x$ is adjacent to two components of $\{K_i|1\leq i\leq s\}$, we set $K^*=K_1+K_2+x+I_1$. Since $H'_1$ satisfied the Theorem by induction hypothesis and $x$ is an $M^*$-unsaturated vertex, $\beta'(G)=\beta'(K^*)+\sum_{i=j}^{s}\beta'(K_i)($j=2$~\text{or}~3)$. Note that $m_\lambda(K_s)=m_\lambda(K_s-u_s)+1$ and $c(G-K_s-u)\leq c(G)-c(K_s)-1$. By using Lemma \ref{GuvH} again, we have
\begin{equation*}
\begin{array}{lll}
m_\lambda(G)&=&m_\lambda(G-K_{s}-u)+m_\lambda(K_s)-1\\
&=&m_\lambda(K^*)+\sum_{i=j}^{s-1}m_\lambda(K_i)+m_\lambda(K_s)-1 \quad ($j=2$~\text{or}~3)\\
&\leq&\beta'(K^*)+c(K^*)+\sum_{i=j}^{s-1}(\beta'(K_i)+c(K_i))+\beta'(K_s)+c(K_s)-1\\
&\leq&\beta'(G)+c(G)-2,
\end{array}
\end{equation*}
which is a contradiction. Therefore, $u$ is not on the cycle of $G$, which implies that $x$ is adjacent to  only one  component of $\{K_i|1\leq i\leq s\}$, namely $K_1$. Let $K_1^*=K_1+x+I_1$. According to Lemma \ref{GuvH}, we have
\begin{equation*}
\begin{array}{lll}
m_\lambda(G)&=&m_\lambda(G-K_s-u)+m_\lambda(K_s)-1\\
&=&m_\lambda(K_1^*)+\sum_{i=2}^{s-1}m_\lambda(K_i)+m_\lambda(K_s)-1\\
&\leq&\beta'(K_1^*)+c(K_1^*)+\sum_{i=2}^{s-1}(\beta'(K_i)+c(K_i))+\beta'(K_s)+c(K_s)-1\\
&\leq&\beta'(G)+c(G)-1.
\end{array}
\end{equation*}
Therefore, $m_\lambda(G)=\beta'(G)+c(G)-1$ if and only if  $m_\lambda(K_1^*)=\beta'(K_1^*)+c(K_1^*)$, $m_\lambda(K_i)=\beta'(K_i)+c(K_i)~(2\leq i\leq s)$, and $\beta'(G)=\beta'(K_1^*)+\sum_{i=2}^{s}\beta'(K_i)$. By Theorem \ref{thm-upper-bound},  $K_1^*$ is isomorphic to $C_3(a, a, a)$ or $C_5$. If $K_1^*\cong C_5$, then each $K_i~(2\leq i\leq s)$ must be either a tree with diameter of $3$ or $C_3(1,1,1)$. Consequently, we can conclude that $\beta'(G)=\beta'(K_1^*)+\sum_{i=2}^{s}\beta'(K_i)+1$, which leads to a contradiction. Therefore, $K_1^*\cong C_3(a,a,a)$, $\beta'(G)=s$, as required.

{\flushleft  \bf Case 2.} $l=2$.

In this case, Suppose $H_2^*=H'_2+x+I_1$. Let $z$ be the only vertex in $H'_1$ adjacent to $x$, and $G-z=H_1^*\cup H_2^*\cup \cdots\cup H_k^*\cup I_3$, where $k\geq2$ for $\beta'(H'_1)\geq2$. Note that $H'_2$ is a tree with a diameter of at most $3$ or $C_5$, or $C_3(a,a,a)$. Therefore, if $H'_2\ncong P_2$ or $\beta'(G)=\beta'(G-z)$, then we can find another induced matching $M^{*'}$ such that there is an unsaturated vertex $y\in V(H'_2)$ in the cycle of $H_2^*$ and $G-y$ has only one non-trivial component, say $K_1$, with $\beta'(K_1)\geq3$. Hence, by induction hypothesis, we can proceed similar as to Case $1$ and obtain the desired result. Next, we only consider the case $H'_2\cong P_2$ and $\beta'(G)=\beta'(G-z)+1$. If $m_\lambda(H_2^*)=m_\lambda(H'_2)+1$, then by applying Lemma \ref{GuvH}, we have
\begin{equation*}
\begin{array}{lll}
m_\lambda(G)&=&m_\lambda(G-H_2^*-z)+m_\lambda(H_2^*)-1\\
&\leq&\sum_{i=1, i\neq2}^k(\beta'(H_i^*)+c(H_i^*))+\beta'(H_2^*)+c(H_2^*)-1\\
&\leq&\beta'(G)+c(G)-2,
\end{array}
\end{equation*}
a contradiction. Next, we only consider the case where $m_\lambda(H_2^*)\leq m_\lambda(H'_2)=\beta'(H'_2)+c(H'_2)=\beta'(H_2^*)+c(H_2^*)-1$.
 Using Lemma \ref{interlacing}, we have
\begin{equation*}
\begin{array}{lll}
m_\lambda(G)&\leq&m_\lambda(G-z)+1\\
&=&\sum_{i=1, i\neq2}^km_\lambda(H_i^*)+m_\lambda(H_2^*)+1\\
&\leq&\sum_{i=1, i\neq2}^k(\beta'(H_i^*)+c(H_i^*))+\beta'(H_2^*)+c(H_2^*)\\
&\leq&\beta'(G)+c(G)-1.
\end{array}
\end{equation*}
Thus, $m_\lambda(G)=\beta'(G)+c(G)-1$ if and only if each inequality involved to become an equality. That is, $m_\lambda(G)=m_\lambda(G-z)+1$, $m_\lambda(H_i^*)=\beta(H_i^*)+c(H_i^*)~(1\leq i\neq2\leq k)$ and $m_\lambda(H_2^*)=\beta'(H_2^*)+c(H_2^*)-1$. Besides, $\beta'(G)=\beta'(G-z)+1$ and $c(G)=c(G-z)$. Since $k\geq2$, there at least one components of $G-z$, say $H_k^*~(k\neq2)$, such that $m_\lambda(H_k^*)=\beta'(H_k^*)+c(H_k^*)$ and has only vertex adjacent to $z$. By Lemma \ref{GuvH}, we have
\begin{equation}\label{thm-3-2-1}
\begin{array}{lll}
m_\lambda(G)&=&m_\lambda(G-z-H_k^*)+m_\lambda(H_k^*)-1\\
&=&\sum_{i=1}^{k}m_\lambda(H_i^*)-1\\
&=&\beta'(G)+c(G)-3,
\end{array}
\end{equation}
which is a contradiction.
\end{proof}

\begin{figure}[htbp]
\centering
\begin{tikzpicture}[x=1.00mm, y=1.00mm, inner xsep=0pt, inner ysep=0pt, outer xsep=0pt, outer ysep=0pt,scale=0.7]
\path[line width=0mm] (74.91,45.30) rectangle +(108.24,89.54);
\definecolor{L}{rgb}{0,0,0}
\definecolor{F}{rgb}{0,0,0}
\path[line width=0.30mm, draw=L, fill=F] (126.80,80.26) circle (1.00mm);
\path[line width=0.30mm, draw=L, fill=F] (80.26,60.52) circle (1.00mm);
\path[line width=0.30mm, draw=L, fill=F] (80.03,49.94) circle (1.00mm);
\path[line width=0.30mm, draw=L, fill=F] (89.43,54.88) circle (1.00mm);
\path[line width=0.30mm, draw=L, fill=F] (100.48,55.11) circle (1.00mm);
\path[line width=0.30mm, draw=L, fill=F] (110.35,60.52) circle (1.00mm);
\path[line width=0.30mm, draw=L, fill=F] (110.11,49.24) circle (1.00mm);
\path[line width=0.30mm, draw=L] (89.66,55.11) -- (80.97,60.52);
\path[line width=0.30mm, draw=L] (89.66,54.88) -- (80.26,49.71);
\path[line width=0.30mm, draw=L] (89.19,54.88) -- (101.18,55.11);
\path[line width=0.60mm, draw=L] (100.24,55.11) -- (110.11,60.52);
\path[line width=0.30mm, draw=L] (100.24,55.11) -- (109.88,49.24);
\path[line width=0.30mm, draw=L, fill=F] (140.43,59.58) circle (1.00mm);
\path[line width=0.30mm, draw=L, fill=F] (140.20,49.00) circle (1.00mm);
\path[line width=0.30mm, draw=L, fill=F] (149.60,53.94) circle (1.00mm);
\path[line width=0.30mm, draw=L, fill=F] (160.64,54.17) circle (1.00mm);
\path[line width=0.30mm, draw=L, fill=F] (170.51,59.58) circle (1.00mm);
\path[line width=0.30mm, draw=L, fill=F] (170.28,48.30) circle (1.00mm);
\path[line width=0.30mm, draw=L] (149.83,54.17) -- (141.14,59.58);
\path[line width=0.30mm, draw=L] (149.83,53.94) -- (140.43,48.77);
\path[line width=0.30mm, draw=L] (149.36,53.94) -- (161.35,54.17);
\path[line width=0.60mm, draw=L] (160.41,54.17) -- (170.28,59.58);
\path[line width=0.30mm, draw=L] (160.41,54.17) -- (170.04,48.30);
\path[line width=0.30mm, draw=L, fill=F] (150.07,115.52) circle (1.00mm);
\path[line width=0.30mm, draw=L, fill=F] (149.83,104.94) circle (1.00mm);
\path[line width=0.30mm, draw=L, fill=F] (159.23,109.88) circle (1.00mm);
\path[line width=0.30mm, draw=L, fill=F] (170.28,110.11) circle (1.00mm);
\path[line width=0.30mm, draw=L, fill=F] (180.15,115.52) circle (1.00mm);
\path[line width=0.30mm, draw=L, fill=F] (179.92,104.24) circle (1.00mm);
\path[line width=0.30mm, draw=L] (159.47,110.11) -- (150.77,115.52);
\path[line width=0.30mm, draw=L] (159.47,109.88) -- (150.07,104.71);
\path[line width=0.30mm, draw=L] (159.00,109.88) -- (170.98,110.11);
\path[line width=0.60mm, draw=L] (170.04,110.11) -- (179.92,115.52);
\path[line width=0.30mm, draw=L] (170.04,110.11) -- (179.68,104.24);
\path[line width=0.30mm, draw=L] (126.56,80.26) -- (89.19,54.64);
\path[line width=0.30mm, draw=L] (127.03,80.73) -- (150.30,53.94);
\path[line width=0.30mm, draw=L, fill=F] (165.34,120.02) circle (1.00mm);
\path[line width=0.30mm, draw=L] (165.11,119.55) -- (158.76,109.21);
\path[line width=0.30mm, draw=L] (165.11,120.72) -- (169.57,110.38);
\path[line width=0.30mm, draw=L, fill=F] (160.41,129.73) circle (1.00mm);
\path[line width=0.30mm, draw=L, fill=F] (170.28,129.96) circle (1.00mm);
\path[line width=0.30mm, draw=L] (159.94,129.96) -- (165.81,119.85);
\path[line width=0.30mm, draw=L] (170.28,130.90) -- (165.81,120.09);
\path[line width=0.30mm, draw=L, fill=F] (148.89,83.08) circle (1.00mm);
\path[line width=0.30mm, draw=L, fill=F] (148.66,72.51) circle (1.00mm);
\path[line width=0.30mm, draw=L, fill=F] (158.06,77.44) circle (1.00mm);
\path[line width=0.30mm, draw=L, fill=F] (169.10,77.68) circle (1.00mm);
\path[line width=0.30mm, draw=L, fill=F] (178.98,83.08) circle (1.00mm);
\path[line width=0.30mm, draw=L, fill=F] (178.74,71.80) circle (1.00mm);
\path[line width=0.30mm, draw=L] (158.29,77.68) -- (149.60,83.08);
\path[line width=0.30mm, draw=L] (158.29,77.44) -- (148.89,72.27);
\path[line width=0.30mm, draw=L] (157.82,77.44) -- (169.81,77.68);
\path[line width=0.60mm, draw=L] (168.87,77.68) -- (178.74,83.08);
\path[line width=0.30mm, draw=L] (168.87,77.68) -- (178.51,71.80);
\path[line width=0.30mm, draw=L, fill=F] (164.17,87.58) circle (1.00mm);
\path[line width=0.30mm, draw=L] (163.93,87.11) -- (157.59,76.77);
\path[line width=0.30mm, draw=L] (163.93,88.29) -- (168.40,77.95);
\path[line width=0.30mm, draw=L, fill=F] (159.23,97.29) circle (1.00mm);
\path[line width=0.30mm, draw=L, fill=F] (169.10,97.53) circle (1.00mm);
\path[line width=0.30mm, draw=L] (158.76,97.53) -- (164.64,87.42);
\path[line width=0.30mm, draw=L] (169.10,98.47) -- (164.64,87.65);
\path[line width=0.30mm, draw=L] (159.47,109.75) -- (127.03,80.60);
\path[line width=0.30mm, draw=L] (158.06,77.31) -- (126.33,80.13);
\path[line width=0.30mm, draw=L, fill=F] (90.60,101.05) circle (1.00mm);
\path[line width=0.30mm, draw=L, fill=F] (79.56,99.88) circle (1.00mm);
\path[line width=0.30mm, draw=L, fill=F] (79.79,95.18) circle (1.00mm);
\path[line width=0.30mm, draw=L, fill=F] (79.56,109.75) circle (1.00mm);
\path[line width=0.30mm, draw=L, fill=F] (79.32,104.81) circle (1.00mm);
\path[line width=0.60mm, draw=L] (79.56,109.51) -- (90.84,101.05);
\path[line width=0.30mm, draw=L] (79.32,104.81) -- (91.31,100.82);
\path[line width=0.30mm, draw=L] (79.56,99.88) -- (91.31,101.29);
\path[line width=0.30mm, draw=L] (79.79,95.41) -- (90.60,101.05);
\path[line width=0.30mm, draw=L] (90.84,101.52) -- (127.74,80.13);
\path[line width=0.30mm, draw=L, fill=F] (89.19,78.02) circle (1.00mm);
\path[line width=0.30mm, draw=L, fill=F] (78.15,76.84) circle (1.00mm);
\path[line width=0.30mm, draw=L, fill=F] (78.38,72.14) circle (1.00mm);
\path[line width=0.30mm, draw=L, fill=F] (78.15,86.71) circle (1.00mm);
\path[line width=0.30mm, draw=L, fill=F] (77.91,81.78) circle (1.00mm);
\path[line width=0.60mm, draw=L] (78.15,86.48) -- (89.43,78.02);
\path[line width=0.30mm, draw=L] (77.91,81.78) -- (89.90,77.78);
\path[line width=0.30mm, draw=L] (78.15,76.84) -- (89.90,78.25);
\path[line width=0.30mm, draw=L] (78.38,72.38) -- (89.19,78.02);
\path[line width=0.30mm, draw=L, fill=F] (91.31,123.14) circle (1.00mm);
\path[line width=0.30mm, draw=L, fill=F] (80.26,121.97) circle (1.00mm);
\path[line width=0.30mm, draw=L, fill=F] (80.50,117.27) circle (1.00mm);
\path[line width=0.30mm, draw=L, fill=F] (80.26,131.84) circle (1.00mm);
\path[line width=0.30mm, draw=L, fill=F] (80.03,126.90) circle (1.00mm);
\path[line width=0.60mm, draw=L] (80.26,131.61) -- (91.54,123.14);
\path[line width=0.30mm, draw=L] (80.03,126.90) -- (92.01,122.91);
\path[line width=0.30mm, draw=L] (80.26,121.97) -- (92.01,123.38);
\path[line width=0.30mm, draw=L] (80.50,117.50) -- (91.31,123.14);
\path[line width=0.30mm, draw=L, fill=F] (119.98,65.09) circle (1.00mm);
\path[line width=0.30mm, draw=L, fill=F] (129.15,64.86) circle (1.00mm);
\path[line width=0.30mm, draw=L, fill=F] (118.10,95.88) circle (1.00mm);
\path[line width=0.30mm, draw=L, fill=F] (126.56,95.88) circle (1.00mm);
\path[line width=0.30mm, draw=L, fill=F] (134.32,95.88) circle (1.00mm);
\path[line width=0.30mm, draw=L] (118.34,96.12) -- (126.80,80.37);
\path[line width=0.30mm, draw=L] (126.56,95.88) -- (126.56,80.37);
\path[line width=0.30mm, draw=L] (134.55,96.12) -- (126.80,79.66);
\path[line width=0.30mm, draw=L] (126.80,80.13) -- (119.51,65.33);
\path[line width=0.30mm, draw=L] (126.80,80.13) -- (129.38,65.09);
\path[line width=0.30mm, draw=L] (91.07,123.38) -- (127.50,79.43);
\path[line width=0.30mm, draw=L] (89.43,77.78) -- (127.03,79.66);
\draw(117,81) node[anchor=base west]{\fontsize{14.23}{17.07}\selectfont $w$};
\end{tikzpicture}%
\caption{A graph with $-2$ as an eigenvalue of multiplicity $\beta'(G)+c(G)-1$.}\label{fig-4}
\end{figure}
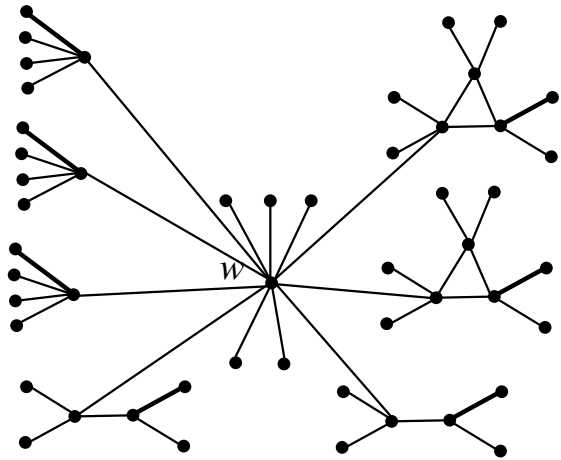

See Fig. $\ref{fig-4} $ for an example.  In this graph,  vertex $w$ has a  degree of $12$, and $G-w$ consists of $7$ non-trivial components. These components include three copies of $K_{1,4}$, two copies of $Y_6$ , and two copies of $C_3(2,2,2)$. It is worth noting that both $K_{1,4}$ and $Y_6$ have a simple eigenvalue of $-2$, and $m_{-2}(C_3(2,2,2))=2$. Therefore, according to Lemma \ref{GuvH}, we have $m_{-2}(G)=8=\beta'(G)+c(G)-1$.

Finally,  we investigate  the relationship between the matching number and $m_\lambda(G)$.
\begin{lem}
Let $G$ be a connected graph with $\lambda \neq 0$ as an eigenvalue of multiplicity $k \geq 1$. Then $m_\lambda(G) \leq \beta(G)+c(G)$, and the equality holds if and only if $G\cong C_3$ or $G\cong K_{1,t}$.
\end{lem}
\begin{proof}
By Theorem \ref{thm-upper-bound}, we have $m_\lambda(G)\leq\beta'(G)+c(G)$. Since $\beta'(G)\leq\beta(G)$,  it follows that $m_\lambda(G)\leq\beta(G)+c(G)$.  Furthermore, equality holds if and only if $\beta'(G)=\beta(G)$.  Referring to Lemma $3.6$ in \cite{Wong} and Theorem \ref{thm-upper-bound},  we can easily observe that the equality holds if and only if $G\cong C_3$ or $G\cong K_{1,t}$.
\end{proof}
\begin{thm}
 Let $G$ be a graph with $\lambda$ as a nonzero eigenvalue. If $\beta'(G)\geq3$ , then $m_\lambda(G) \leq \beta(G)+c(G)-1$, with equality if and only if there is a positive integer $t$ such that $\lambda^2=t$, and there is a vertex $w$ of $G$ such that $G-w$ is the union of $s+1 \geq 3$ copies of $K_{1,t}$ and $H_i+w=K_{1,t+1}$ for each component $H_i$ of $G-w$.
\end{thm}
\begin{proof}
By Corollary \ref{m=b+c-1}, we have $m_\lambda(G)\leq\beta'(G)+c(G)-1$. Since $\beta'(G)\leq\beta(G)$, it follows that $m_\lambda(G)\leq\beta(G)+c(G)-1$.  Moreover, equality holds if and only if $\beta'(G)=\beta(G)$. According to Theorems \ref{extremal-graph} and $3.7$ in \cite{Wong}, the result follows.
\end{proof}


\begin{thebibliography}{}
\small{

\bibitem{CV}D. Cvetkovi\'{c}, M. Drago\v{s}, I. Gutman,
The algebraic multiplicity of the number zero in the spectrum of a bipartite graph,
Mat. Vesnik 9(24) (1972) 141--150.\vspace{-0.35cm}

\bibitem{Cv1}D. Cvetkovi\'{c}, P. Rowlinson, S. Simi\'{c}, A study of eigenspaces of graphs, Linear Algebra Appl.182 (1993) 45--66.\vspace{-0.35cm}

\bibitem{Cv2}D. Cvetkovi\'{c}, P. Rowlinson,  S. Simi\'{c}, An introduction to the Theory of Graph Spectra, Cambridge University Press, Cambridge, 2010.\vspace{-0.35cm}

\bibitem{CGW}Q. Chen, J. Guo, Z. Wang, The multiplicity of a Hermitian eigenvalue on graphs, arXiv preprint arXiv:2306.13882, (2023).\vspace{-0.35cm}

\bibitem{Ma} X. Ma, X. Fang, An improved lower bound for the nullity of a graph in terms of matching number, Linear Multilinear Algebra 68  (2020) 1983--1989.\vspace{-0.35cm}

\bibitem{Row} P. Rowlinson, On multiple eigenvalues of trees, Linear Algebra Appl. 432 (11) (2010) 3007--3011.\vspace{-0.35cm}

\bibitem{Song}Y. Song, X. Song, B. Tam, A characterization of graphs with nullity $|V(G)|-2 m(G)+2 c(G)$, Linear Algebra Appl. 465 (2015) 363--375.\vspace{-0.35cm}

\bibitem{Wong} D. Wong, J. Wang, J. H. Yin, A relation between multiplicity of nonzero eigenvalues of trees and their matching numbers, Linear Algebra Appl. 660 (2023) 80--88.\vspace{-0.35cm}

\bibitem{Wang1} L. Wang, Characterization of graphs with given order, given size and given matching number that minimize nullity, Discrete Math. 339 (2016) 1574--1582.\vspace{-0.35cm}

\bibitem{Wang2} L. Wang, D. Wong, Bounds for the matching number, the edge chromatic number and the independence number of a graph in terms of rank, Discrete Appl. Math. 166 (2014) 276--281.\vspace{-0.35cm}

\bibitem{Zhou} Q. Zhou,  D. Wong, F. Tian, Relation between the nullity of a graph and its matching number, Discrete Appl. Math.  313 (2022) 93--98.\vspace{-0.35cm}
}

\end{thebibliography}
\end{document}